\documentclass[a4paper,12pt]{article}
\usepackage{amssymb,amsthm,amsmath,latexsym}

\newtheorem{thm}{Theorem}[section]
\newtheorem{pro}[thm]{Proposition}
\newtheorem{lem}[thm]{Lemma}

\def\l{\langle}
\def\r{\rangle}

\date{}

\title{\normalsize{\bf AN ENGEL CONDITION FOR ORDERABLE GROUPS}}

\author{\small{\textsc{Pavel Shumyatsky\begin{footnote}{This work was carried out during the first author's visit to the University of Salerno. He would like to thank the Department of Mathematics for hospitality and GNSAGA and CNPq - Brazil for support.}\end{footnote}}}\\
\small{Department of Mathematics, University of Brasilia}\\
\small{Brasilia-DF, 70910-900 Brazil}\\
\small{E-mail: pavel@unb.br}\\
[10pt]
\small{\textsc{Antonio Tortora} and \textsc{Maria Tota}}\\
\small{Dipartimento di Matematica, Universit\`a di Salerno}\\
\small{Via Giovanni Paolo II, 132 - 84084 - Fisciano (SA), Italy}\\
\small{E-mail: antortora@unisa.it, mtota@unisa.it}}

\begin{document}
\maketitle

\begin{abstract} Let $m,n$ be positive integers, $v$ a multilinear commutator word and $w=v^m$. We prove that if $G$ is an orderable group in which all $w$-values are $n$-Engel, then the verbal subgroup $v(G)$ is locally nilpotent. We also show that in the particular case where $v=x$ the group $G$ is nilpotent (rather than merely locally nilpotent).\\

\end{abstract}

\section{Introduction}

An element $x$ of a group $G$ is called a (left) Engel element if for any $g\in G$ there exists $n=n(x,g)\geq 1$ such that $[g,_n x]=1$; as usual, the commutator $[g,_n x]$ is defined recursively by the rule
$$[g,_n x]=[[g,_{n-1} x],x]$$
assuming $[g,_0 x]=g$. If $n$ can be chosen independently of $g$, then $x$ is a {\em (left) $n$-Engel element}. A group $G$ is called an $n$-Engel group if all elements of $G$ are $n$-Engel. It is a long-standing problem whether any $n$-Engel group is locally nilpotent. Following Zelmanov's solution of the restricted Burnside problem \cite{ze1,ze2}, Wilson proved that this is true if $G$ is residually finite \cite{w}. In \cite{kr} Kim and Rhemtulla extended Wilson's theorem by showing that any locally graded $n$-Engel group is locally nilpotent. Recall that a group is locally graded if every nontrivial finitely generated subgroup has a proper subgroup of finite index. In particular, every residually finite group is locally graded. Kim and Rhemtulla also proved that any orderable $n$-Engel group is nilpotent (\cite{kr2}, see also \cite{LM}). A group $G$ is called {\em orderable} if there exists a full order relation $\leq$ on the set $G$ such that $x\leq y$ implies $axb\leq ayb$ for all $a,b,x,y\in G$, i.e. the order on $G$ is compatible with the product of $G$. Recently groups with $n$-Engel word-values were considered \cite{BSTT,STT}.

Let $w$ be a group-word and $G$ be a group. The verbal subgroup $w(G)$ of $G$ is the subgroup generated by all $w$-values in $G$. The words considered in this paper are {\it multilinear commutators}, also known under the name of outer commutator words. These are words that have a form of a multilinear Lie monomial, i.e., they are constructed by nesting commutators but using always different variables. For example the word
$$[[x_1,x_2],[y_1,y_2,y_3],z]$$
is a multilinear commutator  while the Engel word
$$[x,y,y,y]$$
is not. It was proved in \cite{STT} that, given positive integers $m,n$ and a multilinear commutator word $v$, if $G$ is a locally graded group in which all values of the word $w=v^m$ are $n$-Engel, then the verbal subgroup $w(G)$ is locally nilpotent. The purpose of the present paper is to prove a related result for orderable groups.

\begin{thm}\label{main}
Let $m,n$ be positive integers, $v$ a multilinear commutator word and $w=v^m$. If $G$ is an orderable group in which all $w$-values are $n$-Engel, then the verbal subgroup $v(G)$ is locally nilpotent.
\end{thm}

We emphasize that unlike the situation with the locally graded groups where the result was that $w(G)$ is locally nilpotent, Theorem \ref{main} states that so is $v(G)$. Of course, in general $w(G)\leq v(G)$. There still remains the question whether, under assumptions of Theorem \ref{main}, the subgroup $v(G)$ is actually nilpotent. This seems to be a complicated problem. We were able to prove the nilpotency of $v(G)=G$ only in the particular case where $v=x$.

\begin{thm}\label{main2}
Let $m,n$ be positive integers and $G$ be an orderable group. If $x^m$ is $n$-Engel, for all $x\in G$, then $G$ is nilpotent.
\end{thm}

In Section 2 we will provide some basic facts on orderable groups while, in Section 3, we will prove our main results. Ultimately, the proofs are based on Zelmanov's solution of the restricted Burnside problem.

\section{On orderable groups}

Most of our notation is standard. In particular, we write $G=\langle x_1,\dots,x_d\rangle$ to mean that the group $G$ is generated by elements $x_1,\dots,x_d$.  As we have already mentioned a group $G$ is orderable if there exists a full order relation $\leq$ on the set $G$ such that $x\leq y$ implies $axb\leq ayb$ for all $a,b,x,y\in G$. If $G$ is a group with a fixed full order $\leq$ compatible with the product, then $(G,\leq)$ is called an {\it ordered} group. For example an infinite cyclic group is orderable. More generally, any torsion-free nilpotent group is orderable \cite[Theorems 1.3.3, 2.2.4]{BMR}. It is also easy to see that any orderable group is torsion-free.

The class of orderable groups is closed under taking subgroups but a quotient of an orderable group is not necessarily orderable \cite[Section 2.1]{BMR}. A subgroup $C$ of an ordered group $(G,\leq)$ is called {\it convex} if $x\in C$ whenever $1\leq x\leq c$ for some $c\in C$. Obviously $\{1\}$ and $G$ are convex subgroups of $G$; and, if $C$ is a convex subgroup, then every conjugate of $C$ is convex. It is also clear that all convex subgroups of an ordered group form, by inclusion, a totally ordered set, which is closed under intersection and union. If $C$ and $D$ are convex subgroups of an ordered group $G$, with $C<D$, and there is not a convex subgroup $H$ of $G$ such that $C<H<D$, we say that the pair $(C,D)$ is a {\it convex jump} in $G$.

A relatively convex subgroup of an orderable group $G$ is a subgroup convex under some order on $G$. Relatively convex subgroups play important roles: in fact, the quotient $G/N$ of an orderable group $G$ is orderable if and only if $N$ is a normal relatively convex subgroup of $G$. If $G$ is an orderable group, then every term $Z_i(G)$ of the upper central series of $G$ is relatively convex \cite[Theorem 2.2.4]{BMR}. Thus $G/Z_i(G)$ is orderable for all $i\geq 1$. This will be used in the proof of the following lemma.

\begin{lem}\label{Zi}
Let $G=\langle x_1,\dots,x_d\rangle$ be an orderable group and $m_1,\ldots,m_d$ be integers. If $N=\langle x_1^{m_1},\dots,x_d^{m_d}\rangle$, then $N\cap Z_i(G)=Z_i(N)$ for all $i\geq 1$. In particular, if $N$ is nilpotent of class $c$, then so is $G$.
\end{lem}

\begin{proof}
First, we remind the reader that for any $x,y\in G$, if $[x,y^{m_j}]=1$, then $x$ and $y$ commute \cite[Lemma 2.5.1 $(i)$]{BMR}. Let us prove that $N\cap Z_i(G)=Z_i(N)$. Clearly, it is enough to prove that $Z_i(N)\leq N\cap Z_i(G)$. We use induction on $i$. Let $i=1$ and $x\in Z(N)$. Since $[x,x_j^{m_j}]=1$ for any $j=1,\ldots, d$, the previous remark yields $[x,x_j]=1$ and so $x\in N\cap Z(G)$. Now assume that $i\geq 2$ and $Z_{i-1}(N)= N\cap Z_{i-1}(G)$. Since $Z_{i-1}(G)$ is relatively convex, we can pass to the quotient $G/Z_{i-1}(G)$. The image of $Z_{i}(N)$ in the quotient is precisely the center of the image of $N$. Therefore the above argument shows that $Z_i(N)= N\cap Z_i(G)$, as required.

Finally, if $N$ is nilpotent of class $c$ then $N\leq Z_c(G)$ and therefore $G=Z_c(G)$. In fact, if $Z_c(G)$ is a proper subgroup of $G$, then $G/Z_c(G)$ is an orderable group with some nontrivial elements of finite order. Since orderable groups are torsion-free, we obtain a contradiction.
\end{proof}

Orders on a group $G$ in which $\{1\}$ and $G$ are the only convex subgroups are very well known. By a result of H${\rm \ddot{o}}$lder \cite[Theorem\ 1.3.4]{BMR}, a group $G$ with such an order is order-isomorphic to a subgroup of the additive group of the real numbers under the natural order. This implies that, if $(C,D)$ is a convex jump of an ordered group, then $C$ is normal in $D$ and $D/C$ is abelian \cite[Lemma 1.3.6]{BMR}.

\section{The main results}

We say that a subset $S$ of a group $G$ is {\it commutator-closed} if $[x,y]\in S$ whenever $x,y\in S$.

\begin{lem}[see Corollary 5 of \cite{BSTT}]\label{G^i}
Let $m\geq 1$ and $G$  a group generated by a  normal commutator-closed set $S$ such that $x^m$ is Engel for all $x\in S$. If $G$ is finitely generated, then each term of the derived series of $G$ is finitely generated as well.
\end{lem}

Let $H$ and $X$ be subgroups of a group $G$. In the sequel, we denote by $H^X$ the smallest subgroup of $G$ containing $H$ and normalized by $X$.

\begin{lem}[see Corollary 2 of \cite{BSTT}]\label{fg}
Let $x$ be an element of a group $G$ and $H$ a finitely generated subgroup. If $x^m$ is Engel for some $m\geq 1$, then $H^{\langle x\rangle}$ is finitely generated.
\end{lem}

\begin{lem}\label{cn}
Let $G$ be an ordered group generated by a set $X$. Suppose that for each $x\in X$ there exist positive integers $m,n$ such that $x^{m}$ is $n$-Engel. If $C$ is a convex subgroup of $G$, then $C$ is normal in $G$.
\end{lem}

\begin{proof}
Suppose that $C$ is not normal in $G$. Since convex subgroups form a chain, we have either $C^x<C$ or $C<C^x$ for some $x\in X$. Without loss of generality, assume $C<C^x$ and let $c^x\in C^x\backslash C$ for a suitable $c\in C$. Then $C^{x^i}<C^{x^{i+1}}$ for any integer $i$. Moreover, by Lemma \ref{fg}, the subgroup $\l c\r^{\l x\r}$ is finitely generated, so that $\l c\r^{\l x\r}=\l c^{x^{i_1}},\ldots, c^{x^{i_k}}\r$ where $i_1,\ldots,i_k$ are integers. We may assume $i_1<i_2<\ldots<i_k$. It follows that $\l c\r^{\l x\r}\leq C^{x^{i_k}}$. Hence $c^{x^{i_{k}+1}}\in C^{x^{i_k}}$  and therefore $c^x\in C$, a contradiction.
\end{proof}

\begin{lem}[see Lemma 2.2 of \cite{STT}]\label{HK}
Let $G$ be a group generated by two finitely generated subgroups $H$ and $K$. Assume that $K=\l x_1,\ldots,x_d\r$, where each $x_i$ is Engel in $G$. If $K$ is nilpotent, then $H^G$ is finitely generated.
\end{lem}

An important family of multilinear commutator words consists of the derived words $\delta_k$, on $2^k$ variables, which are defined by
\[
\delta_0=x_1 \text{ and }
\delta_k=[\delta_{k-1}(x_1,\ldots,x_{2^{k-1}}),\delta_{k-1}(x_{2^{k-1}+1},\ldots,x_{2^k})]
\text{ for $k\ge 1$.}
\]
The verbal subgroup corresponding to the word $\delta_k$ is the familiar $k$-th derived subgroup of $G$ usually denoted by $G^{(k)}$.

\begin{lem}[see Lemma 4.1 of \cite{shu3}]\label{H^y} Let $G$ be a group and $v$ be a multilinear commutator word. Then there exists $k\geq 1$ such that every $\delta_k$-value in $G$ is a $v$-value.
\end{lem}

In any group $G$ there exists a unique maximal normal locally nilpotent subgroup $F(G)$ (called the Hirsch-Plotkin radical) containing all normal locally nilpotent subgroups of $G$ \cite[12.1.3]{Rob}. In general, $F(G)$ is a subset of the set $L(G)$ of all (left) Engel elements \cite[12.3.2]{Rob}. However, it coincides with $L(G)$ whenever $G$ is soluble (Gruenberg, \cite[12.3.3]{Rob}), or $G$ has an ascending series with locally nilpotent factors (Plotkin, \cite[Exercise 12.3.7]{Rob}).\\[10pt]
\indent The next result is an immediate consequence of Proposition 14 of \cite{BSTT}.

\begin{pro}\label{gamma}
Let $m,n$ be positive integers, $v$ a multilinear commutator word and $w=v^m$. Let $G$ be a residually finite group in which all $w$-values are $n$-Engel. Then the Hirsch-Plotkin radical $F(G)$ is precisely the set of Engel elements $L(G)$.
\end{pro}

\begin{pro}\label{delta}
Let $k,m,n$ be positive integers and $w=\delta_k^m$. If $G$ is an ordered group in which all $w$-values are $n$-Engel, then $G^{(k)}$ is locally nilpotent.
\end{pro}

\begin{proof} Denote by $S$ the set of all $\delta_k$-values in $G$ and choose a finitely generated subgroup $V$ of $G^{(k)}$. Clearly, there exist finitely many $\delta_k$-values $v_1,\dots,v_d$ such that $V\leq\langle v_1,\dots,v_{d}\rangle$. Set $W=\langle w_1,\dots,w_{d}\rangle$, where $w_i=v_i^m$ for any $i\in\{1,\ldots,d\}$. By Lemma \ref{Zi}, the result will be proved once it is shown that $W$ is nilpotent. If $d=1$, then $W$ is cyclic. Thus, we assume that $d\geq 2$ and use induction on $d$.

Let $H=\l w_1,\dots,w_{d-1}, v_d\r$. Suppose that $Q$ is a soluble quotient of $H$. Since $w_1,\dots,w_{d}$ are Engel elements, their images in $Q$ must lie in the Hirsch-Plotkin radical $F(Q)$. It follows that $Q/F(Q)$ is cyclic of order at most $m$. Therefore $F(Q)$ has finite index and hence it is finitely generated. We now conclude that $F(Q)$ is nilpotent. Thus, every soluble quotient of $H$ is polycyclic because it is an extension of a nilpotent group by a cyclic group of order dividing $m$. Set $N=\l v_d\r^H$. By the induction hypothesis the subgroup $\l w_1,\dots,w_{d-1}\r$ is nilpotent and so $H/N$ is nilpotent. On the other hand, by Lemma \ref{HK}, $N$ is finitely generated. More precisely, $N$ is generated by finitely many $\delta_k$-values. Furthermore, $S\cap N$ is a normal commutator-closed subset of $N$. Thus, by Lemma \ref{G^i}, $N^{(i)}$ is finitely generated for every $i$. As a consequence, we obtain that $H^{(i)}$ is also finitely generated for every $i$. Indeed, $H/N^{(i)}$ is a soluble group and therefore it is polycyclic. Hence, $H^{(i)}/N^{(i)}$ is finitely generated and so is $H^{(i)}$.

Let $R$ be the intersection of all convex subgroups $C$ of $H$ such that $H/C$ is soluble. Notice that $R$ is convex and, by Lemma \ref{cn},  any convex subgroup in $H$ is normal. Therefore $R$ is normal in $H$ and $H/R$ is residually soluble. Since every soluble quotient of $H$ is polycyclic and every polycyclic group is residually finite, we deduce that $H/R$ is residually finite. By Proposition \ref{gamma} all Engel elements of $H/R$ belong to the Hirsch-Plotkin radical $F(H/R)$. Therefore $H/R$ is an extension of a nilpotent group by a cyclic group of finite order. In particular $H/R$ is soluble and $WR/R$ is nilpotent. If $R=1$, we are done. Assume $R\neq 1$. Then $R$ contains some $H^{(i)}$, which is finitely generated by the previous argument. Taking into account that $H/H^{(i)}$ is polycyclic, we conclude that $R/H^{(i)}$ is also finitely generated. Hence $R$ is finitely generated and so, by Zorn's Lemma, there exists a maximal convex subgroup $D$ of $H$ contained properly in $R$. We already know that $D$ is normal in $H$. In addition, $(D,R)$ is a jump in the set of all convex subgroups of $H$, hence $R/D$ is an abelian group \cite[Lemma 1.3.6 $(ii)$]{BMR}. It follows that $H/D$ is soluble and $R\leq D$. This is a contradiction since $D<R$.
\end{proof}

\begin{proof}[\bf Proof of Theorem \ref{main}]
By Lemma \ref{H^y}, there exists $k\geq 1$ such that every $\delta_k$-value is a $v$-value. Choose an order on $G$. Therefore, by Proposition \ref{delta}, the subgroup $G^{(k)}$ is locally nilpotent.  Since  $w(G)G^{(k)}/G^{(k)}$ is soluble and generated by Engel elements, it is locally nilpotent. Thus $w(G)$ is an extension of a locally nilpotent group, $w(G)\cap G^{(k)}$, by a locally nilpotent group. It follows that $w(G)$ is also locally nilpotent. Finally, applying Lemma \ref{Zi}, we conclude that $v(G)$ is locally nilpotent.
\end{proof}

As usual, we denote by $\gamma_i(G)$ the $i$-th term of the lower central series of a group $G$. We will  require the following result, due to Burns and Medvedev \cite{BM}.

\begin{thm}\label{BM}
Let $n\geq 1$. There exist constants $c=c(n)$ and $e=e(n)$ depending only on $n$ such that, if $G$ is a finite $n$-Engel group, then the exponent of $\gamma_c(G)$ divides $e$.
\end{thm}

The proof of  Theorem \ref{main2} will now be pretty short.

\begin{proof}[\bf Proof of Theorem \ref{main2}]  According to Theorem \ref{main} $G$ is locally nilpotent.
Let $H$ be an arbitrary finitely generated subgroup of $G$. Thus, $H$ is nilpotent. Let $c$ and $e$ be as in Theorem \ref{BM}. Let $p$ be a prime which does not divide the product $e\cdot m$ and let $P$ be any finite quotient of $H$ having a $p$-power order. Since $p$ does not divide $m$, it follows that every element of $P$ can be written in the form $x^m$ for a suitable element $x\in P$. Thus, $P$ is $n$-Engel and therefore the exponent of $\gamma_c(P)$ divides $e$. Using the fact that $p$ does not divide $e$, we conclude that $\gamma_c(P)=1$. Since $H$ is a finitely generated torsion-free nilpotent group, $H$ is a residually finite $p$-group \cite[Gruenberg, 5.2.21]{Rob} and so $\gamma_c(H)=1$. This holds for any finitely generated subgroup $H$ of $G$ and of course this implies that $\gamma_c(G)=1$. Hence $G$ is nilpotent of class at most $c-1$.
\end{proof}

\end{document}